\documentclass[11pt,review]{my-elsarticle}

\usepackage{lineno,hyperref}
\usepackage{amssymb}
\usepackage{moreverb,multirow}
\usepackage{amssymb}
\usepackage{amsmath,amsthm}
\usepackage{amsfonts,dsfont,mathtools}
\usepackage[margin=2cm]{geometry}
\usepackage[hypcap]{caption}
\usepackage{caption}
\usepackage{subcaption}
\usepackage{multicol}
\usepackage{morefloats}
\usepackage{tikz}
\usetikzlibrary{shapes,arrows,positioning}
\usepackage{booktabs,tabularx}
\usepackage{multirow}
\usepackage{url}
\usepackage{graphics, graphicx, epstopdf}
\usepackage{algorithm}

\renewcommand{\&}{and}

\usepackage{float}

\theoremstyle{plain}
\newtheorem{theorem}{Theorem}[section]

\theoremstyle{definition}

\theoremstyle{remark}

\theoremstyle{remark}

\theoremstyle{plain}
\newtheorem{lemma}{Lemma}[section]

\newtheorem{definition}{Definition}[section]

\journal{This preprint is published in 'Modeling Earth Systems and Environment' 
[\url{https://doi.org/10.1007/s40808-024-02046-8}].}

\begin{document}
	
\begin{frontmatter}
		
\title{Modelling the dynamics of online food delivery services\\ 
on the spread of food-borne diseases}

\author[address11]{Emmanuel Addai}
\ead{eaddai1@ualr.edu; papayawewit@gmail.com}
\address[address11]{College  of Computer and Information Science, 
University of Arkansas at Little Rock,
72204 Little Rock, Arkansas, USA}

\author[address22,address222]{Delfim F. M. Torres}
\ead{delfim@ua.pt; delfim@unicv.cv}
\address[address22]{Center for Research and Development
in Mathematics and Applications (CIDMA),\\ 
Department of Mathematics, University of Aveiro, 3810-193 Aveiro, Portugal}
\address[address222]{Research Center in Exact Sciences (CICE),
Faculty of Sciences and Technology (FCT),\\
University of Cape Verde (Uni-CV),
7943-010 Praia, Cabo Verde}

\author[mysecondaryaddress2]{Zalia Abdul-Hamid}
\address[mysecondaryaddress2]{Department of Logistics and Sustainable Transport, 
University of Maryland, 20740 College Park, Maryland, USA}
\ead{zaliaabdulhamid96@gmail.com}

\author[address11]{Mary Nwaife Mezue}
\ead{Mezuemary@gmail.com}

\author[address3,address1]{Joshua~Kiddy~K.~Asamoah\corref{mycorrespondingauthor}}
\ead{jkkasamoah@knust.edu.gh}
\address[address3]{Department of Mathematics, Saveetha School of Engineering SIMATS, Chennai, India}
\address[address1]{Department of Mathematics, Kwame Nkrumah University of Science and Technology, Kumasi, Ghana}
\cortext[mycorrespondingauthor]{Corresponding Author}

\begin{abstract}
We propose and analyze a deterministic mathematical model
for the transmission of food-borne diseases in a
population consisting of humans and flies.
We employ the Caputo operator to examine the impact of governmental
actions and online food delivery services on the transmission
of food-borne diseases. The proposed model investigates important aspects
such as positivity, boundedness, disease-free equilibrium,
basic reproduction number and sensitivity analysis. The existence and uniqueness of a solution
for the initial value problem is established using Banach and
Schauder type fixed point theorems. Functional techniques are employed
to demonstrate the stability of the proposed model under the Hyers-Ulam condition.
For an approximate solution, the iterative fractional order Predictor-Corrector
scheme is utilized. The simulation of this scheme is conducted using Matlab
as the numeric computing environment, with various fractional order values
ranging from 0.75 to 1. Over time, all compartments demonstrate convergence
and stability. The numerical simulations highlight the necessity for the
government to implement the most effective food safety control interventions.
These measures could involve food safety awareness and training campaigns
targeting restaurant managers, staff members involved in online food delivery,
as well as food delivery personnel.	
\end{abstract}

\begin{keyword}
Mathematical modelling
\sep Food-borne diseases transmission
\sep Online food delivery
\sep Caputo fractional derivatives
\sep Numerical simulations.

\MSC[2020]{26A33; 34D20; 37M05; 92B05.}
\end{keyword}

\end{frontmatter}	

% ---------------------------------------------

\section{Introduction}

Urbanisation, non-communicable diseases, unhealthy diets, and climate change
are all acknowledged as major threats to world health \citep{MS1}. In order to
promote healthy lives and well-being and create inclusive, secure, and
sustainable cities by 2030, the United Nations has asked all of its member
states to take action on 17 Sustainable Development Goals (SDGs) spanning
economic, social, and environmental dimensions \citep{MS2}. Online food delivery
services (OFDS) have the potential to impede our progress towards the SDGs by
affecting our eating, working, and environmental practises. OFDS provide
delivery of a wide range of takeaway foods and beverages from kitchens to
doorsteps and are described as platform-to-consumer delivery operations 
of ready-to-consume meals \citep{MS3}. These include of dishes such as pizza, 
burgers, sandwiches, wraps, spaghetti, and pides. Currently, the Online Food 
Delivery Service (OFDS) industry is largely controlled by major international 
corporations with a global footprint. The practice of ordering ready-to-eat meals 
online introduces new and additional risks to food safety, compounded by the growing 
reach and prevalence of OFDS. These risks stem from the involvement of delivery personnel 
in the supply chain, acting as intermediaries between food producers and the end consumers. 
For example, when the consumer is not present at the point of sale, such as a restaurant, 
it opens opportunities for fraudulent activities like swapping or mislabeling products, 
aiming to profit by substituting high-quality 
items with substandard alternatives \citep{MS4}.

Over the past few decades, concerted efforts by food suppliers, industries, 
and regulatory bodies have been instrumental in elevating food production standards 
to meet rigorous health requirements. This has been pivotal in curbing the public 
health risks and financial repercussions associated with foodborne illnesses. Numerous 
initiatives and programs have been put in place globally by both governmental entities 
and the private sector to realize this objective. Despite these advances, foodborne diseases 
triggered by various agents like bacteria (E. coli, Salmonella, Campylobacter, Listeria, 
Clostridium perfringens), viruses (Norovirus), natural toxins, and chemical contaminants-continue 
to pose significant health threats internationally. This is due to the continuous 
emergence of new hazards while existing ones are being managed.

Recently, the study of infectious disease dynamics has expanded 
significantly, developing into a complex, interdisciplinary domain 
that includes epidemiology, public health, and beyond \cite{NY1}, \cite{NY2}. 
This area brings together a wide range of academic disciplines such as sociology, 
machine learning, artificial neural networks, mathematics, and biology to understand 
and predict the spread of diseases and various social behavioral issues 
\cite{NY3}--\cite{NY5}. Mathematical models are pivotal in epidemiology, 
which focuses on how diseases spread. These models apply mathematical principles 
to depict the dissemination of diseases and their effects on populations. 
They are essential tools for research, enabling scientists to formulate 
and test hypotheses, assess quantitative theories, address specific questions, 
evaluate the effects of parameter changes, and calculate vital parameters 
using mathematical modeling and computer simulations \citep{MS5,MS6,MS7}.
For many years, mathematical models have helped researchers and public health
practitioners to predict the potential impact of disease outbreaks, identify
the key drivers of outbreaks, and evaluate the effectiveness of different
control strategies. Although much research based on mathematical modelling has
been done contributing to fight food-borne diseases, the most worrying aspect about it
is that we are yet far-reaching from the accurate and high degree of prediction
\citep{MS8,MS9,MS10}. As a result, it is critical to improve our
understanding of the dynamics of food-borne transmission by incorporating novel
methods into the current models. Insights into the spread and management of
numerous viral infectious diseases around the world have long been obtained
using mathematical modelling \citep{MS11,MS12,MS13}.
The study conducted by \cite{mugabi2024behaviours} employed a mathematical model 
to investigate how the behaviours of honeybees can decrease the likelihood 
of outbreaks of deformed wing virus in colonies afflicted with Varroa destructor. 
The paper by \cite{abidemi2023optimal} investigated the optimal and cost-effective 
regulation of drug addiction among students using mathematical modelling. The paper by 
\cite{alade2023mathematical} introduced a mathematical model that describes 
the dynamics of Chikungunya virus within a host, taking into account the 
adaptive immune response. The SDIQR mathematical model of COVID-19 is utilised 
by LADT to examine numerical data regarding infected migrants in Odisha. 
The COVID-19 model in the work of \cite{sahu2023sdiqr} utilises analytical power 
series and LADT to estimate the solution profiles of dynamical variables. 
The research undertaken by \cite{peter2022transmission} centres on the creation 
and examination of a deterministic mathematical model for the monkeypox virus. 
The criteria for achieving both local and global stability of disease-free 
and endemic equilibria are established. The evidence suggests that the model 
undergoes a backward bifurcation, where the disease-free equilibrium, which 
is stable in the local context, coexists with an endemic equilibrium. 
The research undertaken by \cite{alla2024mathematical} presented a mathematical 
model that combines the memory effect to mimic the transmission of HIV in 
a heterosexual population. The population is divided into two age cohorts: 
the youth cohort, comprising those aged 15--24, and the adult cohort, 
comprising individuals aged 25 and above. The purpose of classifying 
the population by age is to identify the most vulnerable group to the virus 
based on their sexual behaviour and produce accurate predictions on HIV transmission. 
A study conducted by \cite{ahmad2024mathematical} devised an optimal control 
system that integrated three strategies (public awareness campaign, 
post-exposure vaccination, and isolation) to determine the most efficient 
combination for substantially lowering or eradicating disease transmission. 
We have proven the existence of optimum control and identified the necessary 
conditions for optimality through the application of Pontryagin's notion.
Mathematical models based on traditional integer-order dynamics have been 
extensively applied in epidemiology due to their success in describing disease 
spread, see \citep{MS14} and the references therein. Nevertheless, it is becoming 
increasingly recognized that these models may not sufficiently account for complex 
phenomena such as hereditary characteristics, long-distance interactions, 
and memory effects, which are prevalent across various scientific and engineering 
disciplines, see \citep{MS15,MS16,MS17,MS18,MS19} and the references therein.

In contrast, models employing fractional-order differential operators are gaining favor 
for their enhanced ability to depict the intricate behaviors of dynamic systems. 
Fractional models are particularly adept at incorporating memory and hereditary attributes, 
offering a more nuanced understanding of disease dynamics. This attribute makes them 
exceptionally relevant for a more accurate simulation of biological processes 
and other fields where these complex dynamics are significant \citep{MS20}.

Fractional order differential problems are now one of the most effective
and practical methods for modelling nonlinear processes that appear in
innumerable applications in real-world settings. Recently, the discipline of
mathematical biology has begun to apply several types of fractional order
derivatives. These operators have recently appear in numerous works,
for instance, in mathematical models for HIV/AIDS \citep{MS36},
Ebola-malaria co-infection models \citep{MS37}, smoke age-specific 
models \citep{MS38}, monkeypox transmission dynamics \citep{MS39},
Middle East Lungs Coronavirus dynamism transmission models \citep{MS40},
COVID-19 models \citep{MS41}, and Hepatitis E disease models \citep{MS42}.
For more applications of fractional differential equations see, e.g.,
\citep{MS43,MS44} and references therein. Furthermore, in the context of food-borne diseases
transmission dynamics, the authors in \citep{MS45} used fractional $q$-Homotopy
analysis transformations to study food-borne disease model and suggested control
policies that help the general public comprehension of the importance of
control parameters in the extinction of the diseases. In fact, this study has
accurately predicted the food-borne diseases transmission dynamics and the
impact of government intervention and online food delivery services on the
outbreak, allowing for a more accurate representation of the underlying disease
dynamics. However, none of them, to the best of our knowledge, have studied the effects of
government intervention and online food delivery services on the dynamics of
food-borne diseases transmission using fractional-order derivatives.

The primary goal of this study is to develop and analyze a fractional-order epidemic model 
using the Caputo definition, to assess the impact of governmental measures and online food 
delivery services on controlling the spread of food-borne diseases. The Caputo fractional-order 
derivative is chosen for its ability to incorporate conventional initial conditions into the model's 
solution, making parameter estimation more precise due to its inherent memory capability.
The structure of the paper is methodically laid out to facilitate understanding and exploration of the model. 
Definitions and necessary preliminary information form the basis of Sect.\ref{sec:2}, setting 
the stage for the subsequent analysis. Sect.~\ref{sec:3} is dedicated to the method of the study, 
the detailed construction of the model pertaining to the transmission of food-borne diseases. 
Further, explores the equilibrium states and the basic reproduction number -- key concepts 
in understanding the potential spread of an infection. The mathematical rigor continues, 
which verifies the unique solutions to the proposed model. We discuss the stability of these 
solutions, an essential aspect of predicting the long-term behavior of the disease transmission. 
The paper then moves into more applied aspects in Sect.~\ref{sec:7}, where the numerical 
techniques used in the study are outlined. In Sect.~\ref{sec:8} we present the results 
of the study by simulations that demonstrate the practical implications of the model. 
In Sect.~\ref{sec:9} we present the discussion of the study. Finally, the paper wraps 
up with conclusions in Sect.~\ref{sec:10}, which summarize the findings and their 
relevance to public health policy and disease control measures.

% ----------------------------------------

\section{Preliminaries}
\label{sec:2}

We recall the necessary definitions
of fractional derivative and fractional integral
and a lemma that is important in our proofs.

\begin{definition}[See \citep{MS46}]
\label{Definition2.1}
Let $\mathcal{M}:\mathbb{R}^{+}\rightarrow \mathbb{R}$ and $\alpha \in (n - 1, n)$,
$n \in \mathbb{N}.$ Then, one can define the Caputo derivative with order $\alpha$
for the function $\mathcal{M}$ by
\begin{equation}
^{C}_{0}D^{\alpha}_{t}\mathcal{M}(t)=\frac{1}{\Gamma(n-\alpha)}\int^{t}_{0}(t-s)
^{n-\alpha-1}\mathcal{M}^{(n)}(s){\rm d}s.
\end{equation}
\end{definition}

\begin{definition}[See \citep{MS46}]
\label{Definition2.2}
The related fractional integral is given by
\begin{equation}
_{0}^{C}I^{\alpha}_{t}\mathcal{M}(t)=\frac{1}{\Gamma(\alpha)}\int^{t}_{0}(t-s)^{
\alpha-1}\mathcal{M}(s){\rm d}s.
\end{equation}
\end{definition}

\begin{lemma}[See \citep{MS47,MS48}]
\label{Lemma2.3}
Consider a function $\mathcal{M}\in C[0,\mathcal{T}]$.
The solution of the fractional differential equation
\begin{equation}
\left\{
\begin{array}{llllllllll}
_{0}^{C}D^{\alpha}_{t}\mathcal{M}(t)=\Phi(t),
\quad t\in[0,\eta], \cr
\mathcal{M}(0) = \mathcal{M}_{0},
\end{array}
\right.
\end{equation}
is given by
\begin{equation}
\mathcal{M}\left( t\right) =\mathcal{M}\left( 0\right) + \frac{1}{\Gamma({\alpha})}
\int ^t_0 \Phi (s,\mathcal{M}(s))(t-s)^{\alpha-1} ds.
\end{equation}
\end{lemma}

% ----------------------------------------

\section{Method}
\label{sec:3}

\subsection{Model Formulation}

In this section, we present the formulation of our model.
For this, we consider both human and flies populations
in a closed homogeneous environment. For a human population
at a time $t$, we consider five compartments:
Susceptible humans $S(t)$; Asymptomatic infected humans $A(t)$;
Symptomatic infected humans $I(t)$; Online food delivery personals $D(t)$.
For the human population, we define the quantity $N_{h}$ (total human population) by
\begin{equation}
\label{eq:Nh}
N_{h}(t) = S(t)+ A(t)+ I(t)+ D(t).
\end{equation}
Now we describe the three compartments of flies:
pupae of flies $P_{f}(t)$, adult flies population $G_{f}(t)$,
and parasitic wasps population $W_{p}(t)$. The flies population,
as a whole, is given by
\begin{equation}
\label{eq:Nf}
N_{f}(t) =P_{f}(t)+ G_{f}(t)+ W_{p}(t).
\end{equation}
We assume that individuals have no permanent recovery. Asymptomatic individuals
who develop temporal resistance to the disease become susceptible again at the
rate $\xi$. The value $\xi=0$ indicates no temporal resistance to the
disease, while  $\xi=1$ indicates 100\% resistance to the disease, which means
that no infection occurs $(I(t)\equiv 0)$. We model local government interventions
by a control $u(t)$, which depends on time $t$ and take
values between zero and one. We assume that there is no local government
intervention when $u(t)=0$, while  $u(t)=1$ represents perfect or maximum
intervention. Considering the interrelationship, we formulate
the following deterministic system of nonlinear differential equations:
\begin{equation}
\label{EQ1}
\begin{cases}
\frac{dS(t)}{dt}=\Pi +\xi A-(1-u)\lambda S-\mu_{h}S, \cr
\frac{dA(t)}{dt}=(1-u)\lambda S-(\mu_{h}+\eta+\xi)A,\cr
\frac{dI(t)}{dt}=\eta A-(\mu_{h}+\delta)I,    \cr
\frac{dD(t)}{dt}=\psi(A+I)-((1-u)+\theta+\gamma)D,  \cr
\frac{dP_{f}(t)}{dt}=\sigma G_{f}-(\rho+\mu_{f})P_{f}-\tau P_{f}W_{p}, \cr
\frac{dG_{f}(t)}{dt}=\rho P_{f}-\mu_{f}G_{f},\cr
\frac{dW_{p}(t)}{dt}=\kappa\tau P_{f}W_{p}-\mu_{f}W_{p},
\end{cases}
\end{equation}
subject to given initial conditions
\begin{gather}
S(0)=S_{0}\geq0, \quad A(0)=A_{0}\geq0,\quad  I(0)=I_{0}\geq0, \quad D(0)=D_{0}\geq0,\\
P_{f}(0)=P_{f_{0}}\geq0, \quad G_{f}(0)=G_{f_{0}}\geq0, \quad W_{p}(0)=W_{p_{0}}\geq0,
\end{gather}
where $\lambda=\frac{\beta(r A+I)D}{N_h}+\vartheta G_{f}$ and $t>0$.
The flow diagram of the model is presented in Fig.~\ref{Fig3},
while the description of the parameters
is presented in Table~\ref{T3}. Table~\ref{T3} 
also contains values that is for the numerical simulations.
% ---------------------------------
\begin{figure}[!ht]
\centering{\includegraphics[width=0.8\textwidth]{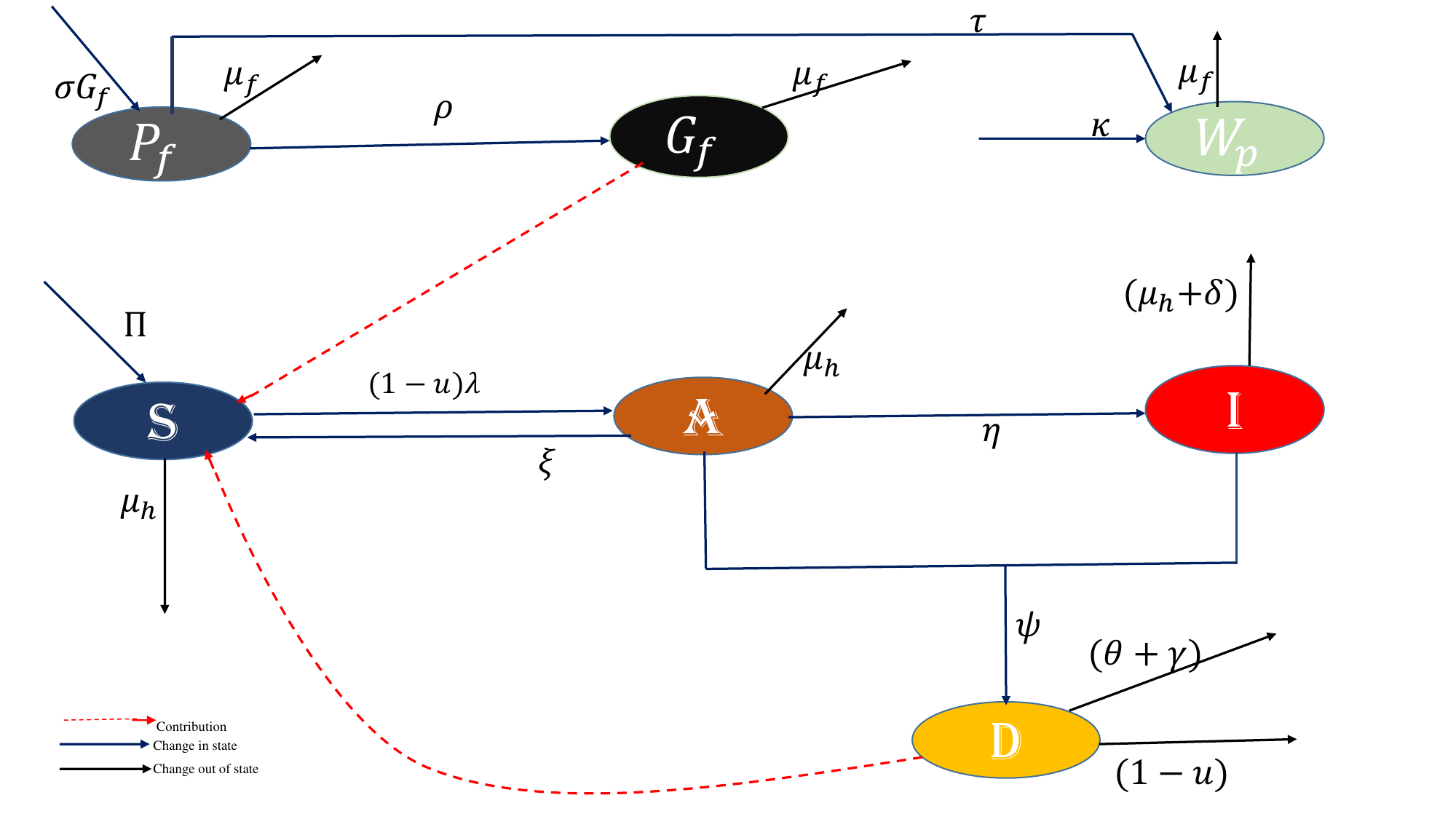}} \\
\caption{Transfer diagram for the dynamic transmission
of food-borne diseases given by model \eqref{EQ1}.}\label{Fig3}
\end{figure}
% ----------------------------------------
\begin{table}[!ht]
\centering\caption{\small Interpretation of the parameters
of model \eqref{EQ1}.}
\label{T3}
\footnotesize
\begin{tabular}{lllll}\hline
Parameter & value &Interpretation & value\\
\hline
$\Pi$&$1000$    & recruitment rate of human& Assumed \\
$\xi$ &$0.0021$ (day$^{-1})$ & temporal resistance rate & Assumed\\
$\beta$ &$0.0014$ (day$^{-1})$& effective contact rate& Assumed \\
$\theta$& 0.50 (day$^{-1}$)& the rate of environmental hygiene & \citep{MS45}\\
$\gamma$& 0.001 (day$^{-1}$)& the rate of $A(t)$ fail to deliver food ordered& \citep{MS45}\\
$r$ &$0.0016667$ (day$^{-1})$& progression rate from $A(t)$ to $I(t)$ & Assumed\\
$\mu_{h}, \mu_{f}$&1/87.7,\,0.000233 (day$^{-1})$& natural death rate for human
and flies & \citep{MS45} \\
$\eta$&0.000375 (day$^{-1})$& the rate at which the asymptomatic become infected & Assumed\\
$\delta$& 0.01 (day$^{-1}$) & disease induced death rate& Assumed\\
$\psi$& 0.47 (day$^{-1}$)& the rate of disease inflow by both $A(t)$ and
$I(t)$ & Assumed\\
$\sigma$& 0.019 (day$^{-1}$)& the rate of adult fly lay egg in the environment & \citep{MS45}\\
$\rho$& 0.003 (day$^{-1}$)& the rate at which $P_{f}(t)$ move to $G_{f}(t)$ & \citep{MS45}\\
$\tau$& 0.0021 (day$^{-1}$)& reduction coefficient of $P_{f}(t)$ and $W_{p}(t)$
due to the interaction & Assumed\\
$\kappa$& 0.01 (day$^{-1}$)& proportionality constant & \citep{MS45}\\
\hline\\
\end{tabular}
\end{table}
% ---------------------------------

According to the explanations of the time-dependent kernel defined by the power law
correlation function, presented in \citep{MS49}, one can take into consideration
the Caputo fractional order derivative on our proposed model:

\clearpage

\begin{equation}
\label{EQ2}
\left\{
\begin{array}{llllllllll}
{}^{C}_{0}D^{\alpha}_{t}S(t)=\Pi +\xi A-(1-u)\lambda S-\mu_{h}S, \cr
{}^{C}_{0}D^{\alpha}_{t}A(t)=(1-u)\lambda S-(\mu_{h}+\eta+\xi)A,\cr
{}^{C}_{0}D^{\alpha}_{t}I(t)=\eta A-(\mu_{h}+\delta)I,    \cr
{}^{C}_{0}D^{\alpha}_{t}D(t)=\psi(A+I)-((1-u)+\theta+\gamma)D,  \cr
{}^{C}_{0}D^{\alpha}_{t}P_{f}(t)=\sigma G_{f}-(\rho+\mu_{f})P_{f}-\tau
P_{f}W_{p}, \cr
{}^{C}_{0}D^{\alpha}_{t}G_{f}(t)=\rho P_{f}-\mu_{f}G_{f},\cr
{}^{C}_{0}D^{\alpha}_{t}W_{p}(t)=\kappa\tau P_{f}W_{p}-\mu_{f}W_{p}, \cr
\end{array}
\right.
\end{equation}
where $t\in[0,\mathcal{T}]$, $\mathcal{T}\in \mathbb{R}$,
and ${}^{C}_{0}D^{\alpha}_{t}$ denotes the Caputo fractional
derivative of order $0<\alpha \leq1$. Note that when
$\alpha = 1$ the fractional-order model \eqref{EQ2}
reduces to \eqref{EQ1}.

% ----------------------------------------

\subsection{Fundamental Qualitative Properties of the Model}

Now, we prove positivity and boundedness
of the solution to the proposed model and we obtain
the corresponding basic reproduction number.

% ----------------------------------------

\subsection{Positivity and boundedness}

Here, we demonstrate the epidemiological significance of the model under
consideration. We prove that the system classes for model \eqref{EQ1}
are non-negative for all $t$, which means that our suggested model
has, for all $t> 0$, non-negative solutions for non-negative initial values.

\begin{lemma}
\label{Lemma1}
Let
\begin{equation}
\mathcal{M}(t)=\{(S(t), A(t), I(t), D(t), P_{f}(t), G_{f}(t), W_{p}(t))\in
\mathbb{R}^7\},
\quad t \geq 0,
\end{equation}
and $\mathcal{T}\in \mathbb{R}$. If the initial data is positive,
that is, $\mathcal{M}(0)>0$, then the solution classes of model
\eqref{EQ1} are non-negative in $[0,\mathcal{T}]$. Also,
\begin{equation}
\lim_{t\rightarrow \infty}N_{h}(t)\leq
\frac{\Pi}{\mu_{h}},
\end{equation}
and
\begin{equation}
\lim_{t\rightarrow \infty}N_{f}(t)\leq
\frac{\sigma}{\mu_{f}},
\end{equation}
with $N_{h}$ and $N_{f}$ given by
\eqref{eq:Nh} and \eqref{eq:Nf}, respectively.
\end{lemma}

\begin{proof}
Take $t_{1}=\sup\{t>0:\mathcal{M}(t)>0\ \forall t \in[0,\mathcal{T}[\}$
and let us consider the first dynamical equation of model \eqref{EQ1}:
\begin{equation}
\frac{dS(t)}{dt}=\Pi +\xi A-(1-u)\lambda S-\mu_{h}S.
\end{equation}
For simplicity, let 
\begin{equation}
\mathcal{Y}=\left[(1-u)\lambda-\mu_{h}\right]
\end{equation}.
Thus,
\begin{equation}
\frac{dS(t)}{dt}+\mathcal{Y}S(t)=\Pi +\xi A.
\end{equation}
This leads to
\begin{equation}
\frac{d}{dt}\left(S(t)\exp\left(\int_{0}^{t_1}\mathcal{Y}(s)ds\right)\right)
=\left(\Pi+\xi A\right)\exp\left(\int_{0}^{t_1}\mathcal{Y}(s)ds\right),
\end{equation}
and
\begin{equation}
S(t)=\exp\left(-\int_{0}^{t_1}\mathcal{Y}(s)ds\right)\left[\int_{0}^{t_1}\left(\Pi
+\xi A\right)\exp\left(\int_{0}^{t_1}\mathcal{Y}(s)ds\right)\right].
\end{equation}
Hence, we proved that $S(0)> 0$ for all $t>0$. Similarly, from the other classes of
model \eqref{EQ1}, we can get $\mathcal{M}(0)>0$ for all $t>0$. Note that after the
summation of the state variables in model \eqref{EQ1}, we have
\begin{equation}
\begin{array}{llllllllll}
\frac{dN_{h}(t)}{dt}=\Pi-\mu_{h}N_{h}(t), \cr
\frac{dN_{f}(t)}{dt}=\sigma-\mu_{f}N_{f}(t).\cr
\end{array}
\end{equation}
Hence, $\lim_{t\rightarrow \infty}N_{h}(t)\leq \frac{\Pi}{\mu_{h}}$
and $\lim_{t\rightarrow \infty}N_{f}(t)\leq \frac{\sigma}{\mu_{f}}$,
which completes the proof.
\end{proof}

For feasibility of model \eqref{EQ1}, we let  $\mathcal{X}
=\Omega_{h}+\Omega_{F}\subset \mathbb{R}^{1}_{+}
\times \mathbb{R}^{2}_{+}$, where
\begin{equation}
\mathcal{X}=\left\{(S,A,I,D,P_f,G_f,W_p)\in \mathbb{R}^{7}_{+}:
N_h(t)\leq\frac{\Pi}{\mu_{h}},~N_{f}(t)\leq\frac{\sigma}{\mu_{f}}\right\}.
\end{equation}
In solving the inequalities above we have
\begin{equation}
\begin{array}{llllllllll}
N_h(t)\leq N_h(0)e^{-\mu_{h}(t)}+\frac{\Pi}{\mu_{h}}\left(1-e^{-\mu_{h}(t)}\right), \cr
N_{f}(t)\leq
N_{f}(0)e^{-\mu_{f}(t)}+\frac{\sigma}{\mu_{f}}(1-e^{-\mu_{f}(t)}).\cr
\end{array}
\end{equation}
Therefore,
\begin{equation}
\label{ineq;Nh:Nf}
\begin{array}{llllllllll}
\limsup_{t\rightarrow\infty}N_h(t)\leq\frac{\Pi}{\mu_{h}}, \cr
\limsup_{t\rightarrow\infty}N_{f}(t)\leq\frac{\sigma}{\mu_{f}}.\cr
\end{array}
\end{equation}
Hence, our proposed model \eqref{EQ1} is considered to be mathematically well posed.

For positivity of model \eqref{EQ2}, we add all the state variables in
model \eqref{EQ2} to get
\begin{equation}
\begin{array}{llllllllll}
&^{C}_{0}D^{\alpha}_{t}N_{h}(t)=\Pi-\mu_{h}N_{h}(t), \cr
&^{C}_{0}D^{\alpha}_{t}N_{f}(t)=\sigma-\mu_{f}N_{f}(t),\cr
\end{array}
\end{equation}
from which we conclude that \eqref{ineq;Nh:Nf} also holds
for \eqref{EQ2}, showing that model \eqref{EQ2} remains
in the feasible region $\mathcal{X}$.

% ----------------------------------------

\subsection{Food-borne diseases reproduction number and the disease-free equilibrium}

For the parameters of our model \eqref{EQ1}, let
$k_{1}=\mu_{h}+\eta+\xi$, $k_{2}=\mu_{h}+\delta$, $k_{3}=1-u +\theta+\gamma$,
and $k_{4}=\rho+\mu_{f}$. Equating the right-hand side of the proposed model
\eqref{EQ1} to zero, that is,
\begin{equation}
\left\{
\begin{array}{llllllllll}
\Pi +\xi A-(1-u)\lambda S-\mu_{h}S=0, \cr
(1-u)\lambda S-(\mu_{h}+\eta+\xi)A=0,\cr
\eta A-(\mu_{h}+\delta)I=0,    \cr
\psi(A+I)-((1-u)+\theta+\gamma)D=0,  \cr
\sigma G_{f}-(\rho+\mu_{f})P_{f}-\tau P_{f}W_{p}=0, \cr
\rho P_{f}-\mu_{f}G_{f}=0,\cr
\kappa\tau P_{f}W_{p}-\mu_{f}W_{p}=0, \cr
\end{array}
\right.
\end{equation}
we obtain that the disease-free equilibrium of model \eqref{EQ1} is given by
\begin{equation}
E_{1}^{0}=\left(\frac{\Pi}{\mu_{h}},0,0,0,\frac{\mu_{f}}{\kappa\tau},
\frac{\rho}{\kappa\tau}, \frac{\sigma\rho-(\rho+\mu_{f})\mu_{f}}{ \tau\mu_{f}}\right),
\end{equation}
or
\begin{equation}
E_{2}^{0}=\left(\frac{\Pi}{\mu_{h}},0,0,0,0,0,0\right),
\end{equation}
while the endemic disease equilibrium point is $(S^{*},
A^{*},I^{*},D^{*},P_{f}^{*},G_{f}^{*},W_{p}^{*})$, where
\begin{equation}
S^{*}=\frac{\Pi
k_{1}}{\lambda(k_{1}-\xi)((1-u)+k_{1}\mu_{h}},~~A^{*}=\frac{\lambda((1-u)S
}{k_{1}},~~I^{*}=\frac{\eta  A }{ k_{2}},~~D^{*}=\frac{\psi(k_{2}+\eta) A }{
k_{2}k_{3}},
\end{equation}
\begin{equation}
P_{f}^{*}=\frac{\mu_{f}}{\kappa\tau},~~G_{f}^{*}=\frac{\rho}{\kappa\tau},~~W_{
p}^{*}=\frac{\rho(\sigma-\mu_{f})-\mu_{f}^{2}}{\tau\mu_{f}}.
\end{equation}
Let $\mathcal{H} = (S,E, I,D, P_{f},G_{f}, W_{p})^{T }$. Then one observes that
\begin{equation}
\frac{d\mathcal{H}}{dt}=\mathcal{F}-\mathcal{V},
\end{equation}
where
\begin{equation}
\mathcal{F}=\left[ \begin{matrix}
0 & 0 & 0 & 0 & 0 & 0 & 0\\
\frac{(1-u)\vartheta\rho}{\kappa\tau} & 0 & 0 & 0 & 0 &
\frac{(1-u)\vartheta\Pi}{\mu_{h}}  & 0  \\
0 & 0 & 0 & 0 & 0 & 0 & 0\\
0 & 0 & 0 & 0 & 0 & 0 & 0\\
0 & 0 & 0 & 0 & 0 & 0 & 0\\
0 & 0 & 0 & 0 & 0 & 0 & 0\\
0 & 0 & 0 & 0 & 0 & 0 & 0\\
\end{matrix} \right],
\end{equation}
\begin{equation}
\mathcal{V}=\left[ 
\begin{matrix}
\frac{(1-u)\vartheta\rho}{\kappa\tau} +\mu_{h} & -\xi & 0 & 0 &
0 & \frac{(1-u)\vartheta\Pi}{\mu_{h}}  & 0\\
0 & k_{1} & 0 & 0 & 0 & 0 & 0\\
0 &  -\eta & k_{2}  & 0 & 0 & 0 & 0\\
0 & -\psi  & -\psi & k_{3}& 0 & 0 & 0 \\
0 & 0 & 0 & k_{4}& -\sigma & \frac{ \mu_{f}}{\kappa} & 0 \\
0 & 0 & 0 & -\rho & \mu_{f} & 0& 0 \\
0 & 0 & 0 & 0 &-\kappa\tau \frac{\sigma\rho-(\rho+\mu_{f})\mu_{f}}{
\tau\mu_{f}} & 0 & 0 \\
\end{matrix} \right].
\end{equation}
The food-borne diseases reproduction number is given by
\begin{equation}
\mathcal{R}_{0}= \rho(\mathcal{FV}^{-1})
\end{equation}
with $\rho(\cdot)$ representing the spectral radius,
from which we obtain
\begin{equation}
\mathcal{R}_{0}=\frac{
\xi\rho\vartheta(1-u)}{k_{1}\rho\vartheta(1-u)+\kappa k_{1}\tau \mu_{h}}.
\end{equation}

% ----------------------------------------

\subsection{Existence and Uniqueness}
\label{sec:5}

We reformulate the right-hand side of model \eqref{EQ2} as follows:
\begin{equation}
\label{EQ3}
\left\{
\begin{array}{llllllllll}
\mathcal{M}_{1}(t,S,A,I,D,P_{f},G_{f},W_{p})=\Pi +\xi A-(1-u)\lambda
S-\mu_{h}S, \cr
\mathcal{M}_{2}(t,S,A,I,D,P_{f},G_{f},W_{p})=(1-u)\lambda
S-(\mu_{h}+\eta+\xi)A,\cr
\mathcal{M}_{3}(t,S,A,I,D,P_{f},G_{f},W_{p})=\eta A-(\mu_{h}+\delta)I, \cr
\mathcal{M}_{4}(t,S,A,I,D,P_{f},G_{f},W_{p})=\psi(A+I)-((1-u)+\theta+\gamma)D,\cr
\mathcal{M}_{5}(t,S,A,I,D,P_{f},G_{f},W_{p})=\sigma
G_{f}-(\rho+\mu_{f})P_{f}-\tau P_{f}W_{p}, \cr
\mathcal{M}_{6}(t,S,A,I,D,P_{f},G_{f},W_{p})=\rho P_{f}-\mu_{f}G_{f},\cr
\mathcal{M}_{7}(t,S,A,I,D,P_{f},G_{f},W_{p})=\kappa\tau
P_{f}W_{p}-\mu_{f}W_{p}.
\end{array}
\right.
\end{equation}
From \eqref{EQ3}, our proposed model \eqref{EQ2} can be expressed as
\begin{equation}
\label{EQ4}
\left\{
\begin{array}{llllllllll}
&^{C}_{0}D^{\alpha}_{t}\mathcal{M}(t)=\Phi(t, \mathcal{M}(t)),
~t\in[0,\eta],~0<\alpha\leq1, \cr
&\mathcal{M}(0) = \mathcal{M}_{0},
\end{array}
\right.
\end{equation}
where
\begin{equation}
\label{EQ5}
\mathcal{M}(t)=
\left\{
\begin{array}{llllllllll}
S(t),\cr
A(t),\cr
I(t),\cr
D(t),\cr
P_{f}(t),\cr
G_{f}(t),\cr
W_{p}(t),\cr
\end{array}
\mathcal{M}_{0}=
\right.
\left\{
\begin{array}{llllllllll}
S(t),\cr
A(0),\cr
I(0),\cr
D(0),\cr
P_{f}(0),\cr
G_{f}(0),\cr
W_{p}(0),\cr
\end{array}
\right.
\end{equation}	
and
\begin{equation}
\label{EQ6}
\Phi(t, \mathcal{M}(t))=
\left\{
\begin{array}{llllllllll}
\mathcal{M}_{1}(t,S,A,I,D,P_{f},G_{f},W_{p}),\\
\mathcal{M}_{2}(t,S,A,I,D,P_{f},G_{f},W_{p}),\\
\mathcal{M}_{3}(t,S,A,I,D,P_{f},G_{f},W_{p}),\\
\mathcal{M}_{4}(t,S,A,I,D,P_{f},G_{f},W_{p}),\\
\mathcal{M}_{5}(t,S,A,I,D,P_{f},G_{f},W_{p}),\\
\mathcal{M}_{6}(t,S,A,I,D,P_{f},G_{f},W_{p}),\\
\mathcal{M}_{7}(t,S,A,I,D,P_{f},G_{f},W_{p}).
\end{array}
\right.
\end{equation}
Using the ideas in Lemma~\ref{Lemma2.3}, system \eqref{EQ4} yields
\begin{equation}
\mathcal{M}(t)=\mathcal{M}_{0}(t)+\frac{1}{\Gamma(\alpha)}\int^{t}_{0}\Phi(s,
\mathcal{M}(s))(t-s)^{\alpha-1}ds.
\end{equation}
Further, we let $ 0\leq t\leq \mathcal{T}$ with the Banach space
$\mathcal{B}=C([0,\mathcal{T}]\times \mathbb{R}^7_+, \mathbb{R}_+)$ under the
norm
\begin{equation}
||\mathcal{M}||_\mathcal{B}=||(S,A,I,D,P_{f},G_{f},W_{p})||_\mathcal{B}
=\sup\{|\mathcal{D}(t)|: t\in \mathcal{T} \},
\end{equation}
such that 
\begin{equation}
|\mathcal{D}|:= |S|+|A|+|I|+|D|+|P_{f}|+|G_{f}|+|W_{p}|.
\end{equation}
For obtaining existence and uniqueness, we assume some growth conditions on
function vector \begin{equation}m:[0,\mathcal{T}]\times
\mathbb{R}^7_+\longrightarrow\mathbb{R}_+\end{equation} as:
$(\digamma_{1})$ $\exists$ $Q_{m}, C_{m}$ such that
\begin{equation}
\Phi(t, \mathcal{M}(t))\leq Q_{m}|\mathcal{M}|+C_{m}.
\end{equation}
$(\digamma_{2})$ $\exists$ $\textbf{L}_{m}>0$ such that
if $\mathcal{M},\tilde{\mathcal{M}}\in\mathcal{B}$, then
\begin{equation}
|\Phi(t, \mathcal{M}(t))-\Phi(t, \tilde{\mathcal{M}}(t))|
\leq \textbf{L}_{m}[|\mathcal{M}-\tilde{\mathcal{M}}|].
\end{equation}

\begin{theorem}
\label{Theorem1}
Under the continuity of ${m}$ together with
$(\digamma_{1})$, system \eqref{EQ4} has at least one solution.
\end{theorem}

\begin{proof}
We shall arrive at the required conclusion using the Schauder
fixed point theorem. Let us take a closed subset 
$\mathcal{Z}$ of $\mathcal{B}$ as
$\textbf{T}_{\Omega}=\{\mathcal{M}\in \mathcal{B}: ||\mathcal{M}||\leq
\mathbb{R}, \mathbb{R}>0\}$, where $\textbf{T}_{\Omega}$ is the operator
defined as $\textbf{T}_{\Omega}:\mathcal{Z}\rightarrow\mathcal{Z}$ such that
\begin{equation}
\textbf{T}_{\Omega}(\mathcal{M})=\mathcal{M}_{0}(t)+\frac{1}{\Gamma(\alpha)}
\int^{t}_{0}\Phi(s, \mathcal{M}(s))(t-s)^{\alpha-1}ds,
\end{equation}
which means that
\begin{equation}
\label{EQ9}
\begin{split}
\begin{array}{llllllllll}
|\textbf{T}_{\Omega}(\mathcal{M})(t)|&\leq|\mathcal{M}_{0}|+\frac{1}{\Gamma(\alpha)}
\int^{t}_{0}|\Phi(s, \mathcal{M}(s))|(t-s)^{\alpha-1}ds,\cr
&\leq|\mathcal{M}_{0}|+\frac{1}{\Gamma(\alpha)}\int^{t}_{0}(t-s)^{\alpha-1}[Q_{m
}|\mathcal{M}|+C_{m}]ds,\cr
&\leq|\mathcal{M}_{0}|+\frac{\mathcal{T}^\alpha}{\Gamma(\alpha+1)}[Q_{m}||
\mathcal{M}||+C_{m}].
\end{array}
\end{split}
\end{equation}
From \eqref{EQ9}, it follows that
$|\textbf{T}_{\Omega}(\mathcal{M})|\leq|\mathcal{M}_{0}|
+\frac{\mathcal{T}^\alpha}{\Gamma(\alpha+1)}[Q_{m}||\mathcal{M}||+C_{m}]$ and also
$\textbf{T}_{\Omega}\in\mathcal{Z}$ such that
$\textbf{T}_{\Omega}(\mathcal{Z})\subset\mathcal{Z}.$  Also it reveals that
the operator $\textbf{T}_{\Omega}$ is bounded. For completely continuity we proceed
as follows. Let $t_{2}<t_{1}\in [0,\mathcal{T}]$ such that
\begin{equation}
\label{EQ10}
\begin{split}
\begin{array}{llllllllll}
|\textbf{T}_{\Omega}(\mathcal{M})(t_{2})|-|\textbf{T}_{\Omega}(\mathcal{M})(t_{1
})|&\leq\big|\frac{1}{\Gamma(\alpha)}\int^{t_{2}}_{0}\Phi(s,
\mathcal{M}(s))(t_{2}-s)^{\alpha-1}ds \cr
&-\frac{1}{\Gamma(\alpha)}\int^{t_{1}}_{0}\Phi(s,
\mathcal{M}(s))(t_{1}-s)^{\alpha-1}ds\big|\cr
&\leq\frac{1}{\Gamma(\alpha)}\big[\int^{t_{1}}_{0}[(t_{1}-s)^{\alpha-1}-(t_{2}-s
)^{\alpha-1}]\Phi(s, \mathcal{M}(s))ds \cr
&+ \int^{t_{2}}_{t_{1}}(t_{2}-s)^{\alpha-1}\Phi(s, \mathcal{M}(s))ds\big],\cr
&\leq\frac{Q_{m}\mathcal{R}+C_{m}}{\Gamma(\alpha+1)}[(t_{2}^{\alpha}
-t_{1}^{\alpha})+2(t_{2}-t_{1})^{\alpha}].
\end{array}
\end{split}
\end{equation}
Basically, we can see from \eqref{EQ10} that
$|\textbf{T}_{\Omega}(\mathcal{M})(t_{2})|-|\textbf{T}_{\Omega}(\mathcal{M})(t_{
1})|\rightarrow 0$
as $t_{2}\rightarrow t_{1}$. Hence, $\textbf{T}_{\Omega}$
is an equicontinuous operator.
With the help of the Arzela Ascoli theorem, we know that function
$\textbf{T}_{\Omega}$ is a completely continuous function and uniformly bounded.
Again, by Schauder's fixed point theorem, we conclude that our proposed system
has at least one solution.
\end{proof}

\begin{theorem}
\label{Theorem2}
Suppose that $(\digamma_{2})$ holds. Then the considered
system \eqref{EQ4} has a unique solution if
\begin{equation}
\frac{\mathcal{T}^\alpha}{\Gamma(\alpha+1)}\textbf{L}_{m}<1.
\end{equation}
\end{theorem}

\begin{proof}
If $\mathcal{M}, \tilde{\mathcal{M}}\in\mathcal{B}$, then
\begin{equation}
\label{EQ11}
\begin{split}
\begin{array}{llllllllll}
||\textbf{T}_{\Omega}(\mathcal{M})-\textbf{T}_{\Omega}(\tilde{\mathcal{M}})||
&\leq \sup_{t\in [0,\mathcal{T}]}\big|\frac{1}{\Gamma(\alpha)}\int^{t}_{0}\Phi(s,
\mathcal{M}(s))(t-s)^{\alpha-1}ds \\
&-\frac{1}{\Gamma(\alpha)}\int^{t}_{0}\Phi(s,
\tilde{\mathcal{M}}(s))(t-s)^{\alpha-1}ds\big|,\cr
&\leq\frac{\mathcal{T}^\alpha}{\Gamma(\alpha+1)}\textbf{L}_{m}||\mathcal{M}
-\tilde{\mathcal{M}}||.
\end{array}
\end{split}
\end{equation}
Hence,
$||\textbf{T}_{\Omega}(\mathcal{M})-\textbf{T}_{\Omega}(\tilde{\mathcal{M}})||
\leq\frac{\mathcal{T}^\alpha}{\Gamma(\alpha+1)}\textbf{L}_{m}||\mathcal{M}
-\tilde{\mathcal{M}}||$,
which completes the proof: from the contraction principle, the operator has a unique
fixed point and, consequently, our proposed model has a unique solution.
\end{proof}

% ----------------------------------------

\subsection{Hyers-Ulam Stability (HU)}
\label{sec:6}

The stability of numerical results of our proposed system will be examined in
this section.

\begin{lemma}
\label{Lemma4}
The solution $\mathcal{M}(t)$ of
\begin{equation}
\label{EQ12a}
\begin{split}
\begin{array}{llllllllll}
^{C}_{0}D^{\alpha}_{t}\mathcal{M}(t)=\Phi(t,\mathcal{M}(t)) +\Upsilon(t), \cr
\mathcal{M}(0) = \mathcal{M}_{0},
\end{array}
\end{split}
\end{equation}
satisfies the relation
\begin{equation}
\label{EQ12b}
\left|\mathcal{M}(t)-\left(\mathcal{M}_{0}(t)+\frac{1}{\Gamma(\alpha)}\int^{t}_{0}
\Phi(s, \mathcal{M}(s))(t-s)^{\alpha-1}ds\right)\right|
\leq \frac{\mathcal{T}^\alpha}{\Gamma(\alpha+1)}\epsilon=\Delta\epsilon.
\end{equation}
\end{lemma}

\begin{proof}
The proof is standard and we omit it here.
\end{proof}

\begin{theorem}
\label{Theorem5}
Suppose that assumption $(\digamma_{2})$ together with
\eqref{EQ10} hold. Then the solution of the integral \eqref{EQ6} is Hyers-Ulam (HU) stable
and, consequently, the numerical results of the considered model are HU stable if
\begin{equation}
\Theta=\frac{\mathcal{T}^\alpha}{\Gamma(\alpha+1)}\textbf{L}_{m}<1.
\end{equation}
\end{theorem}

\begin{proof}
Let $\mathcal{M}\in\mathcal{B}$ be any solution and
$\tilde{\mathcal{M}}\in \mathcal{B}$
be the unique solution of (\ref{EQ6}). Then,
\begin{equation}
\label{EQ14}
\begin{split}
\begin{array}{llllllllll}
|\mathcal{M}(t)-\tilde{\mathcal{M}}(t)|
&=\left|\mathcal{M}(t)-\left(\mathcal{M}_{0
}(t)+\frac{1}{\Gamma(\alpha)}\int^{t}_{0}\Phi(s,
\tilde{\mathcal{M}}(s))(t-s)^{\alpha-1}ds\right)\right| \cr
&\leq\left|\mathcal{M}(t)-\left(\mathcal{M}_{0}(t)+\frac{1}{\Gamma(\alpha)}
\int^{t}_{0}\Phi(s, \tilde{\mathcal{M}}(s))(t-s)^{\alpha-1}ds\right)\right| \cr
&+ \left|\frac{1}{\Gamma(\alpha)}\int^{t}_{0}\Phi(s,
\tilde{\mathcal{M}}(s))(t-s)^{\alpha-1}ds-\frac{1}{\Gamma(\alpha)}\int^{t}_{0}
\Phi(s, \tilde{\mathcal{M}}(s))(t-s)^{\alpha-1}ds\right| \cr
&\leq\Delta\epsilon+\Theta|\mathcal{M}-\tilde{\mathcal{M}}|.
\end{array}
\end{split}
\end{equation}
Thus, $$|\mathcal{M}-\tilde{\mathcal{M}}|\leq\frac{\Delta}{1-\Theta}\epsilon.$$
Hence, we conclude that system (\ref{EQ6}) is HU stable.
\end{proof}

% ----------------------------------------

\section{Numerical Scheme}
\label{sec:7}

The fractional predictor-corrector approach, which was established and examined
for its convergence and error bounds in \citep{MS50}, is a numerical explicit
technique used to test the performance of the suggested fractional-order model
\eqref{EQ1} under the Caputo differential operator. Let us consider
\begin{equation}
\left\{
\begin{array}{llllllllll}
^{C}_{0}D^{\alpha}_{t}\mathcal{M}(t)=\Phi(t,\mathcal{M}(t)), \cr
\mathcal{M}(0) = \mathcal{M}_{0}.
\end{array}
\right.
\end{equation}
Choose the step length $h=\frac{T}{M}$, where $M$ is a positive integer and $T$ is
the upper limit of the closed interval of integration $[0, T]$.
Using the integral equation equivalent to system \eqref{EQ2},
$\mathcal{M}_{a}(t_{j+1})$, $j=0,1,\ldots,n$, can be calculated by
\begin{equation*}
\mathcal{M}_{a}(t_{j+1})
=\frac{h^{\alpha}}{\Gamma(\alpha+2)}\left[
\sum^{n}_{j=0} d_{j,n+1} \Phi(t_{j},
\mathcal{M}_{a}(t_{j}))+
\Phi(t_{n+1}, \mathcal{M}^{p}_{a}(t_{n+1}))\right]
+ \mathcal{M}_{0},
\end{equation*}
where
\begin{equation*}
d_{j,n+1}=
\begin{cases}
\begin{array}{l}
n^{\alpha+1}-(n-\alpha)(n+1)^{\alpha},~~~ j=0,\\
(n-j+2)^{\alpha+1}+(n-j)^{\alpha+1}-2(n-j+1)^{\alpha+1}, ~~~1\leq j\leq n,\\
1,~~~j=n+1.
 \end{array}
\end{cases}
\end{equation*}
The predictor formula is derived as follows:
\begin{equation*}
\begin{split}
 \mathcal{M}^{p}_{a}(t_{n+1})= \frac{1}{\Gamma(\alpha+1)}\sum^{n}_{j=0}
h^{\alpha} [(n-j+1)^{\alpha}-(n-j)^{\alpha}]\Phi(t_{j},
\mathcal{M}_{a}(t_{j}))+ \mathcal{M}_{0}.
\end{split}
\end{equation*}
Thus the corrector formula for system \eqref{EQ2} is
\begin{equation*}
\begin{split}
& ^{C}_{0}D^{\alpha}_{t}S(t_{j+1})=
\frac{\mathbf{h}^{\alpha}}{\Gamma(\alpha+2)}\left[\sum^{n}_{j=0} d_{j,n+1}
\Phi_{1}(t_{j}, \mathcal{M}_{a}(t_{j}))+
\Phi_{1}(t_{n+1}, \mathcal{M}^{p}_{a}(t_{n+1}))\right]+ S_{0},\\
& ^{C}_{0}D^{\alpha}_{t}A(t_{j+1})=
\frac{\mathbf{h}^{\alpha}}{\Gamma(\alpha+2)}\left[\sum^{n}_{j=0} d_{j,n+1}
\Phi_{2}(t_{j}, \mathcal{M}_{a}(t_{j}))+
\Phi_{2}(t_{n+1}, \mathcal{M}^{p}_{a}(t_{n+1}))\right]+ A_{0},\\
& ^{C}_{0}D^{\alpha}_{t}I(t_{j+1})=
\frac{\mathbf{h}^{\alpha}}{\Gamma(\alpha+2)}\left[\sum^{n}_{j=0} d_{j,n+1}
\Phi_{3}(t_{j}, \mathcal{M}_{a}(t_{j}))+
\Phi_{3}(t_{n+1}, \mathcal{M}^{p}_{a}(t_{n+1}))\right]+ I_{0},\\
& ^{C}_{0}D^{\alpha}_{t}D(t_{j+1})=
\frac{\mathbf{h}^{\alpha}}{\Gamma(\alpha+2)}\left[\sum^{n}_{j=0} d_{j,n+1}
\Phi_{4}(t_{j}, \mathcal{M}_{a}(t_{j}))+
\Phi_{4}(t_{n+1}, \mathcal{M}^{p}_{a}(t_{n+1}))\right]+ D_{0},\\
& ^{C}_{0}D^{\alpha}_{t}P_{f}(t_{j+1})=
\frac{\mathbf{h}^{\alpha}}{\Gamma(\alpha+2)}\left[\sum^{n}_{j=0} d_{j,n+1}
\Phi_{5}(t_{j}, \mathcal{M}_{a}(t_{j}))+
\Phi_{5}(t_{n+1}, \mathcal{M}^{p}_{a}(t_{n+1}))\right]+ P_{f_{0}},\\
& ^{C}_{0}D^{\alpha}_{t}G_{f}(t_{j+1})=
\frac{\mathbf{h}^{\alpha}}{\Gamma(\alpha+2)}\left[\sum^{n}_{j=0} d_{j,n+1}
\Phi_{6}(t_{j}, \mathcal{M}_{a}(t_{j}))+
\Phi_{6}(t_{n+1}, \mathcal{M}^{p}_{a}(t_{n+1}))\right]+ G_{f_{0}},\\
& ^{C}_{0}D^{\alpha}_{t}W_{p}(t_{j+1})=
\frac{\mathbf{h}^{\alpha}}{\Gamma(\alpha+2)}\left[\sum^{n}_{j=0} d_{j,n+1}
\Phi_{7}(t_{j}, \mathcal{M}_{a}(t_{j}))+
\Phi_{7}(t_{n+1}, \mathcal{M}^{p}_{a}(t_{n+1}))\right]+ W_{p_{0}}.
\end{split}
\end{equation*}

The results of this iterative scheme are given in Section~\ref{sec:8}.

% ----------------------------------------

\section{Results}
\label{sec:8}

For the purpose of validating our created iterative scheme, we now present
some numerical simulations. For this, we start by assuming initial values for
each compartment of our proposed model \eqref{EQ1}, thus: $S= 500000$; $A= 300000$;
$I= 3500$; $D= 2000$; $P_{f}= 250000$; $G_{f}= 200000$; $W_{p}= 2000$. 
Figure~\ref{Fig3**} shows the Latin hypercube sampling of the parameters 
in the control reproduction number. It is noticed that the following parameters 
(temporal resistance rate, the rate at which pupae of flies move to adult flies, 
and infectivity rate of adult flies) contribute positively to the spread 
of food-borne diseases and increase in these parameters (local government interventions, 
the rate at which the asymptomatic become infected, proportionality constant, 
reduction coefficient of pupae of flies and parasitic wasps due to the interaction, 
and natural death rate for human ) reduces the spread of food-borne diseases.  We have employed
the Predictor-Corrector scheme of Section~\ref{sec:7} to obtain a
numerical solution to the system. We
compare the effects of various fractional order values with of a step size 0.2
throughout the time range [0,300] against the suitable parameter values listed
in Table~\ref{T3}. We deduced from all of the graphical simulations that the Caputo
fractional derivative efficiently describes the intricate dynamics of the
presented model of food-borne diseases.

% ------------------------

\subsection{Sensitivity analysis}

Now, we utilize sensitivity analysis to determine the relative significance 
of each model parameter in the reproduction number of food-borne diseases, 
denoted as $\mathcal{R}_{0}$, by referring to the data provided in Table~\ref{T3}. 
The current objective is to ascertain the influence of alterations in the model 
parameters on the reproduction number of food-borne diseases. 
Similar instances may be found in the work of 
\citep{asamoah2021sensitivity,asamoah2021non,asamoah2023fractionals}. 
We generate a three-dimensional graph representing the natural mortality rate 
for humans and flies, the constant of proportionality, the rate of temporal resistance, 
the reduction coefficient of $P_f(t)$ and $W_p(t)$ resulting from their interaction, 
the rate at which $P_f(t)$ transitions to $G_f(t)$, and the intervention by local government, 
presented in Fig.~\ref{A}. In addition, the Latin Hypercube Sampling (LHS) technique was utilized 
to generate 2000 samples from a uniform distribution. These samples were used to determine 
the global sensitivity of the different generic factors in the food illnesses reproduction number. 
The resulting sensitivity analysis is depicted in Fig.~\ref{Fig3**}. 
\begin{figure}[!ht]
\centering
\begin{subfigure}{.4\textwidth}
\centering
\includegraphics[scale=0.5]{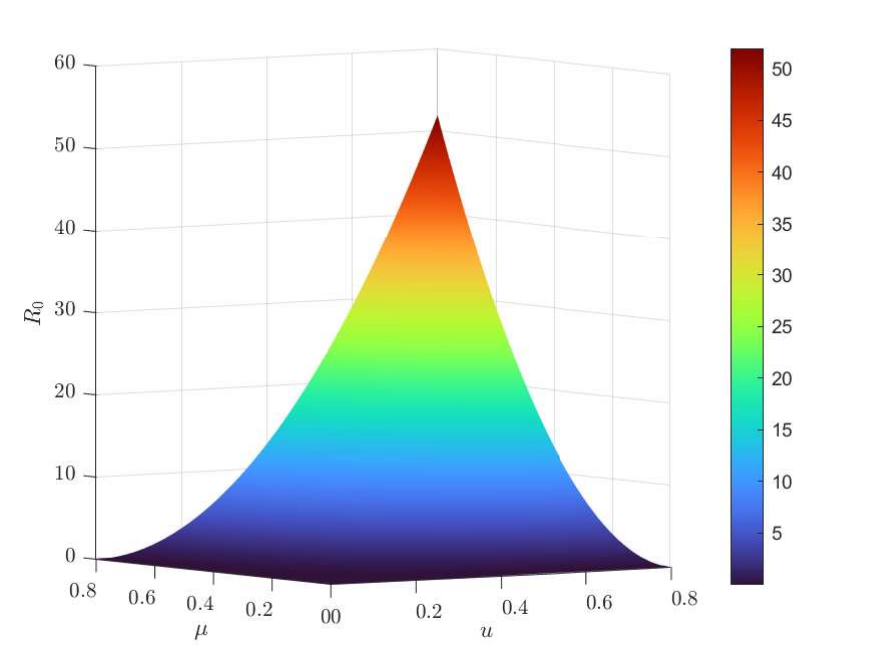}
\caption{$\mathcal{R}_{0}$ against $\mu$ and $u$}\label{F1a}
\end{subfigure}
\begin{subfigure}{.4\textwidth}
\centering
\includegraphics[scale=0.5]{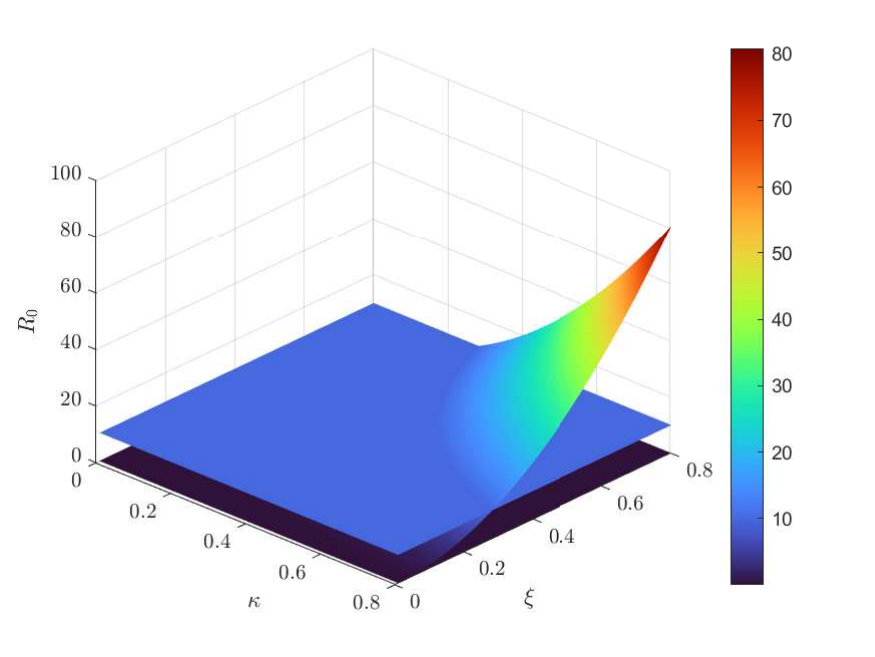}
\caption{$\mathcal{R}_{0}$ against $\kappa$ and $\xi$}\label{F2a}
\end{subfigure}
\begin{subfigure}{.4\textwidth}
\centering
\includegraphics[scale=0.5]{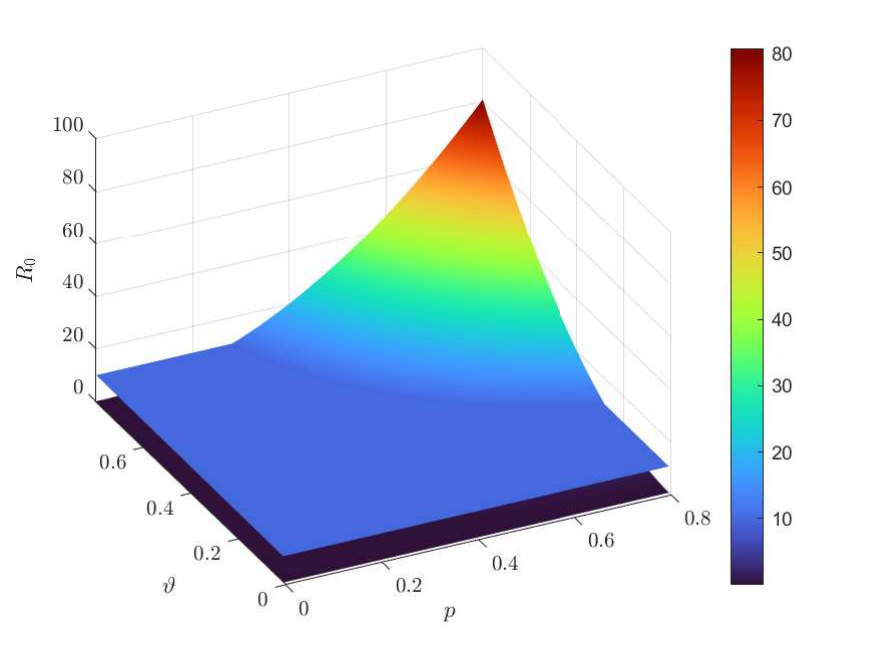}
\caption{$\mathcal{R}_{0}$ against $\vartheta$ and $p$}\label{F3a}
\end{subfigure}
\begin{subfigure}{.4\textwidth}
\centering
\includegraphics[scale=0.5]{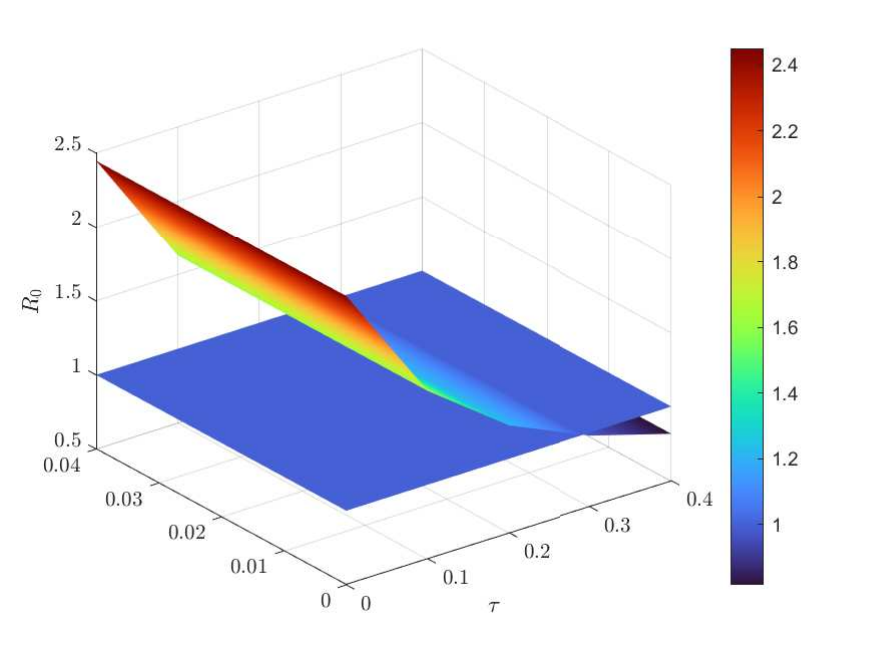}
\caption{$\mathcal{R}_{0}$ against $\tau$}\label{F4a}
\end{subfigure}
\caption{Dynamics of $\mathcal{R}_{0}$ on various parameters.}\label{A}
\end{figure}
\begin{figure}[!ht]
\centering{\includegraphics[width=0.5\textwidth]{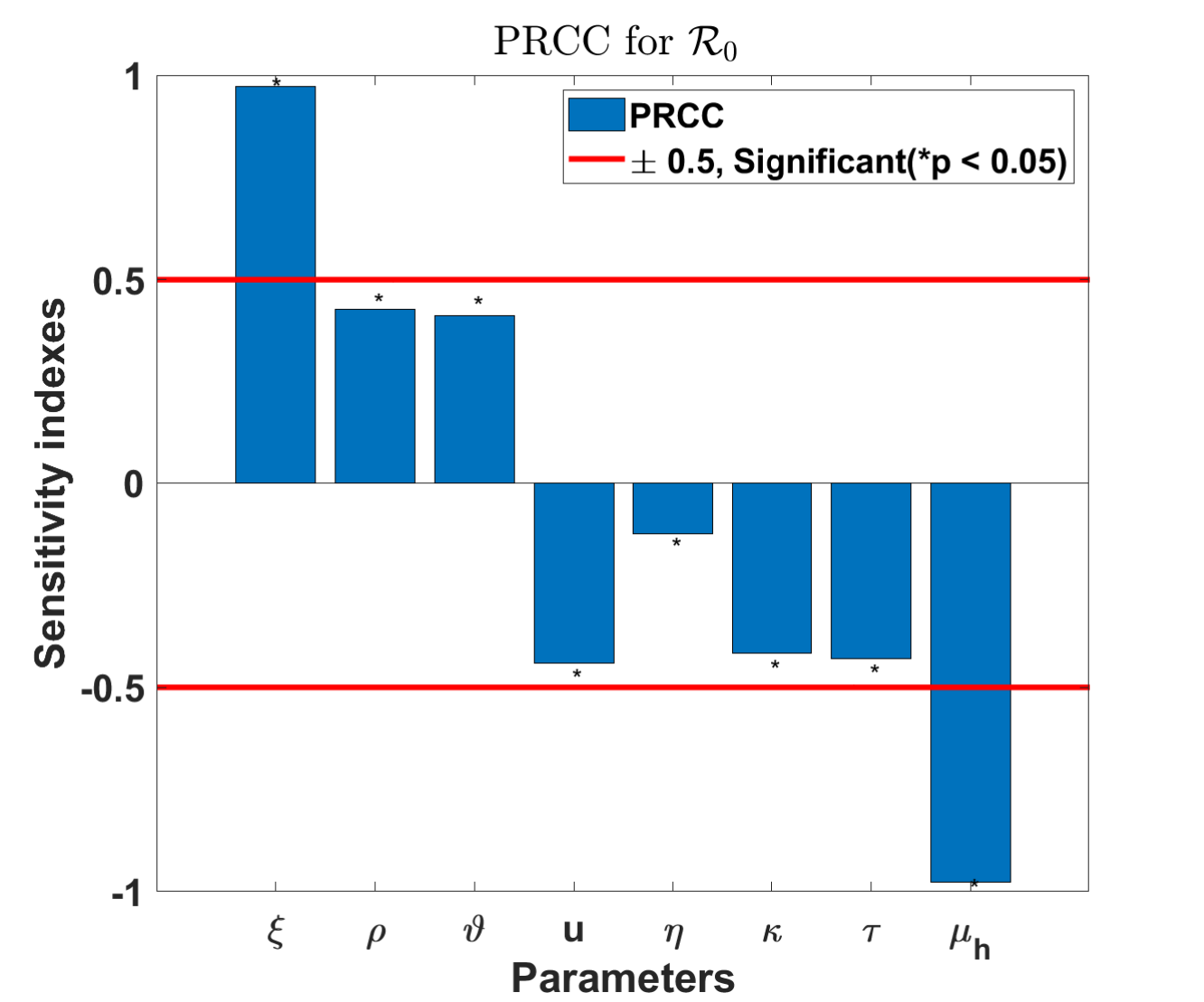}} \\
\caption{Sensitivity plot}\label{Fig3**}
\end{figure}

\begin{figure}[!ht]
\centering
\begin{subfigure}{.4\textwidth}
\centering
\includegraphics[scale=0.5]{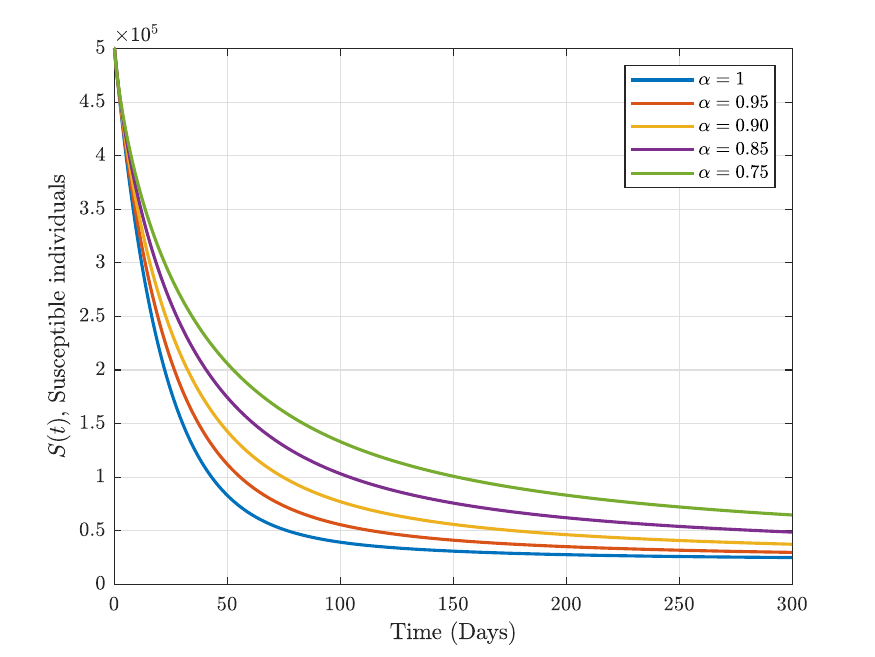}
\caption{}\label{F1}
\end{subfigure}
\begin{subfigure}{.4\textwidth}
\centering
\includegraphics[scale=0.5]{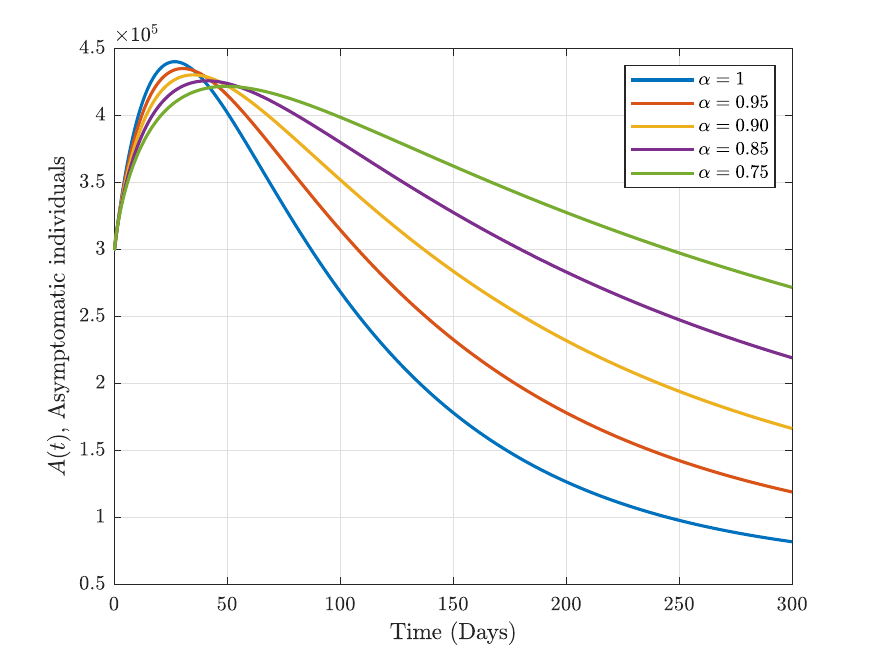}
\caption{}\label{F2b}
\end{subfigure}
\begin{subfigure}{.4\textwidth}
\centering
\includegraphics[scale=0.5]{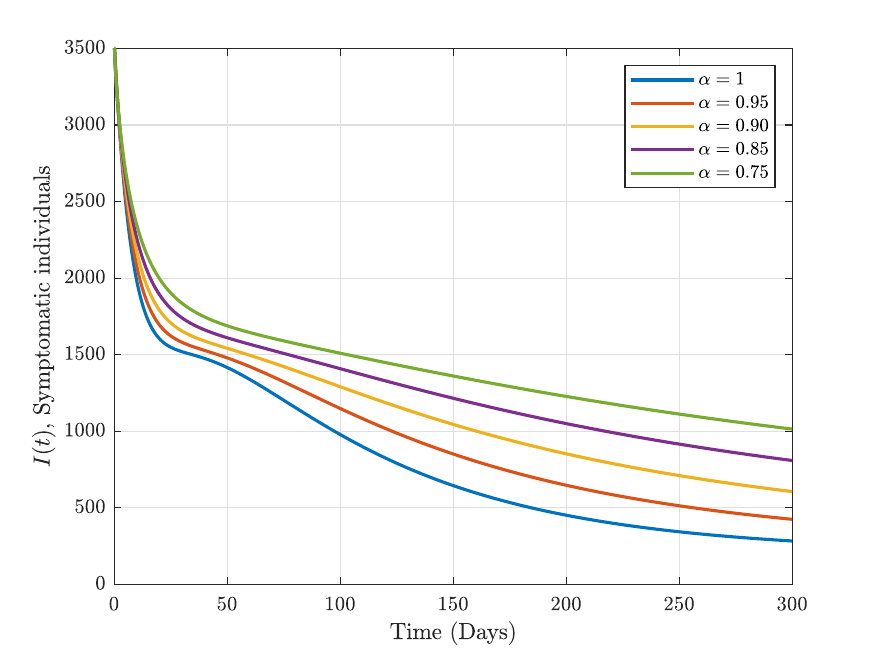}
\caption{}\label{F3b}
\end{subfigure}
\begin{subfigure}{.4\textwidth}
\centering
\includegraphics[scale=0.5]{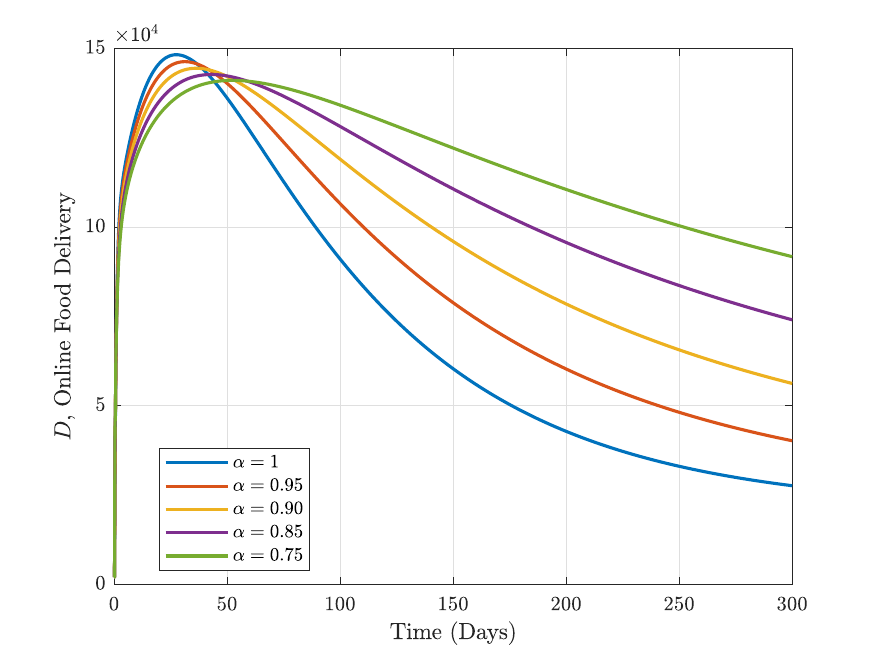}
\caption{}\label{F4}
\end{subfigure}
\caption{Numerical trajectory of food-borne diseases transmission under
Caputo fractional operator.}\label{AA}
\end{figure}
\begin{figure}[!ht]
\centering
\begin{subfigure}{.4\textwidth}
\centering
\includegraphics[scale=0.5]{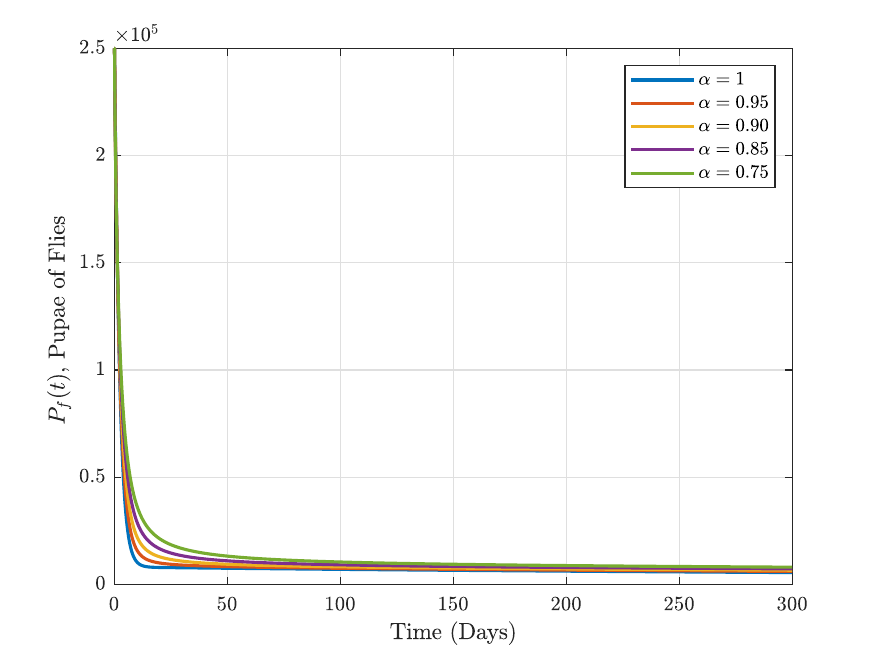}
\caption{}\label{F5}
\end{subfigure}
\begin{subfigure}{.4\textwidth}
\centering
\includegraphics[scale=0.5]{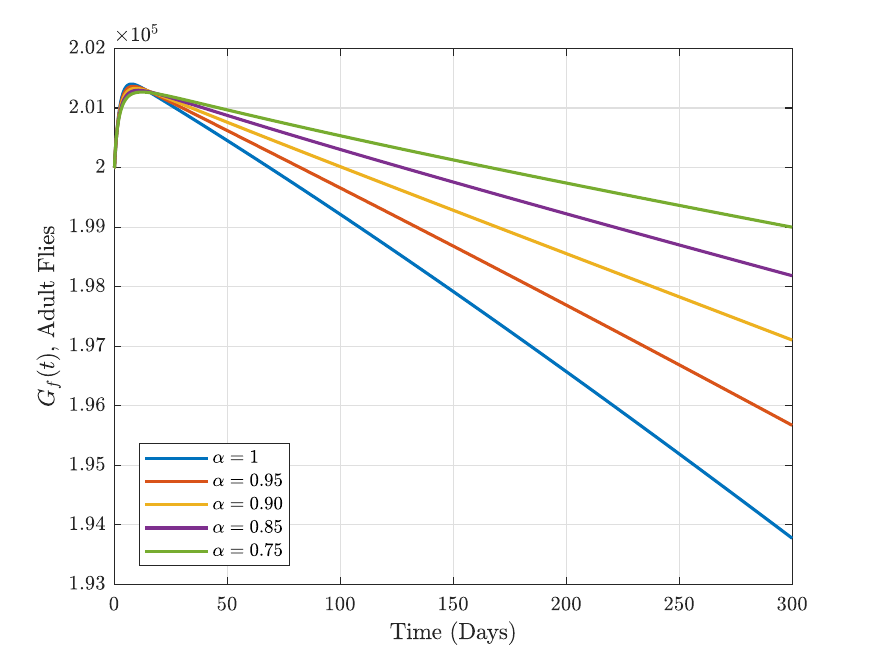}
\caption{}\label{F6}
\end{subfigure}
\begin{subfigure}{.4\textwidth}
\centering
\includegraphics[scale=0.55]{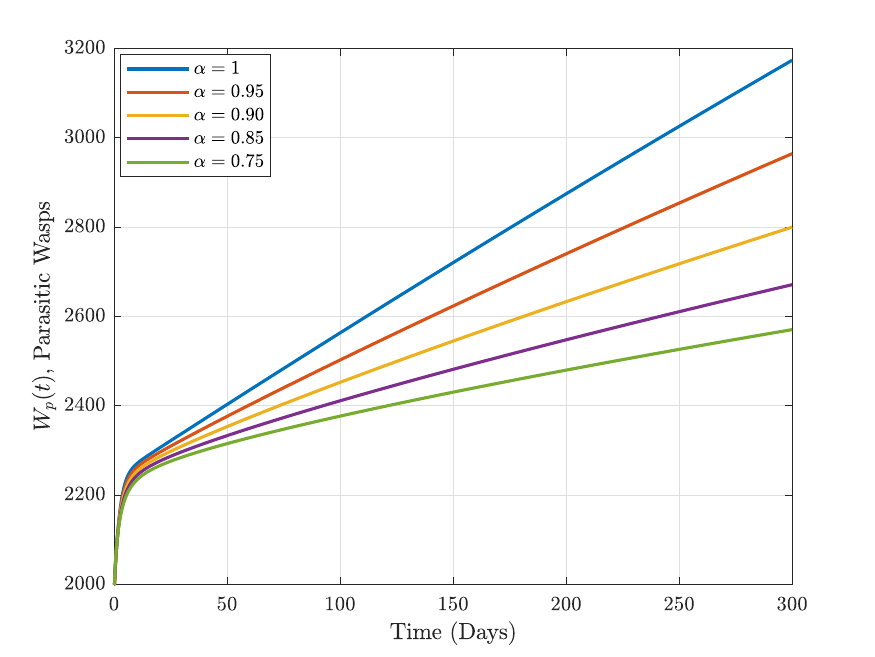}
\caption{}\label{F7}
\end{subfigure}
\caption{Numerical trajectory of food-borne diseases transmission under
the Caputo fractional operator.}\label{B}
\end{figure}

\begin{figure}[!ht]
\centering
\begin{subfigure}{.4\textwidth}
\centering
\includegraphics[scale=0.5]{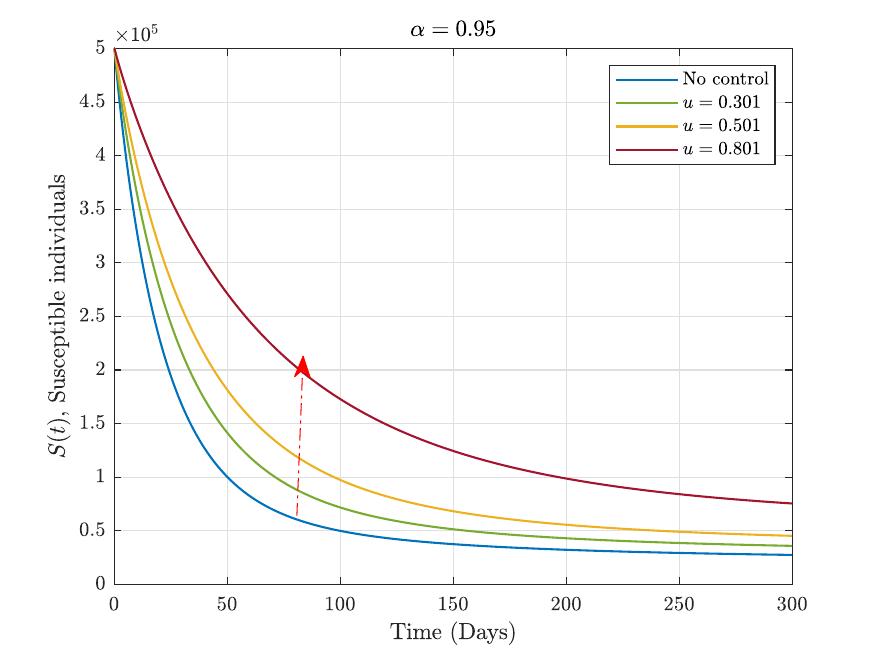}
\caption{}\label{F8}
\end{subfigure}
\begin{subfigure}{.4\textwidth}
\centering
\includegraphics[scale=0.5]{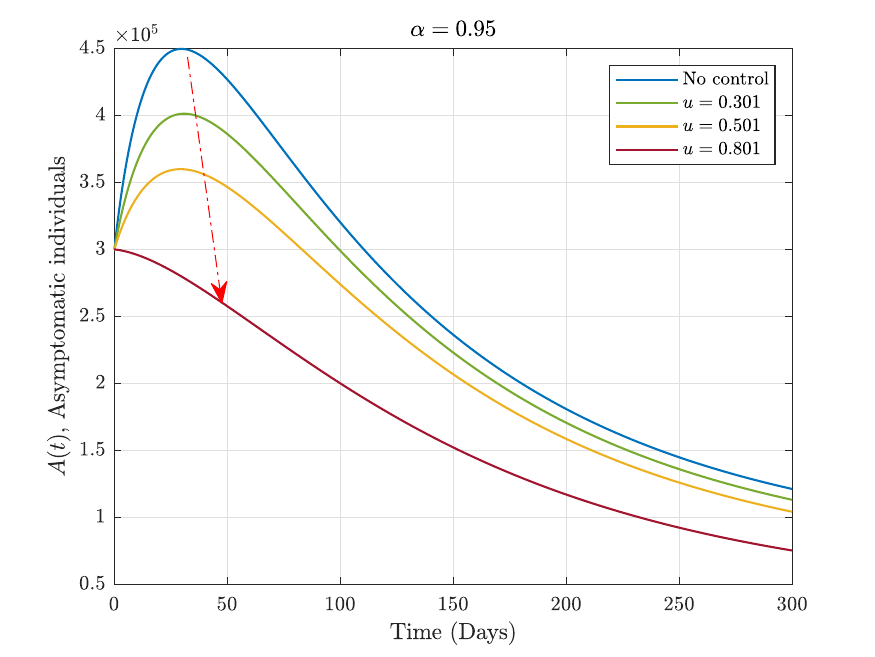}
\caption{}\label{F9}
\end{subfigure}
\begin{subfigure}{.4\textwidth}
\centering
\includegraphics[scale=0.5]{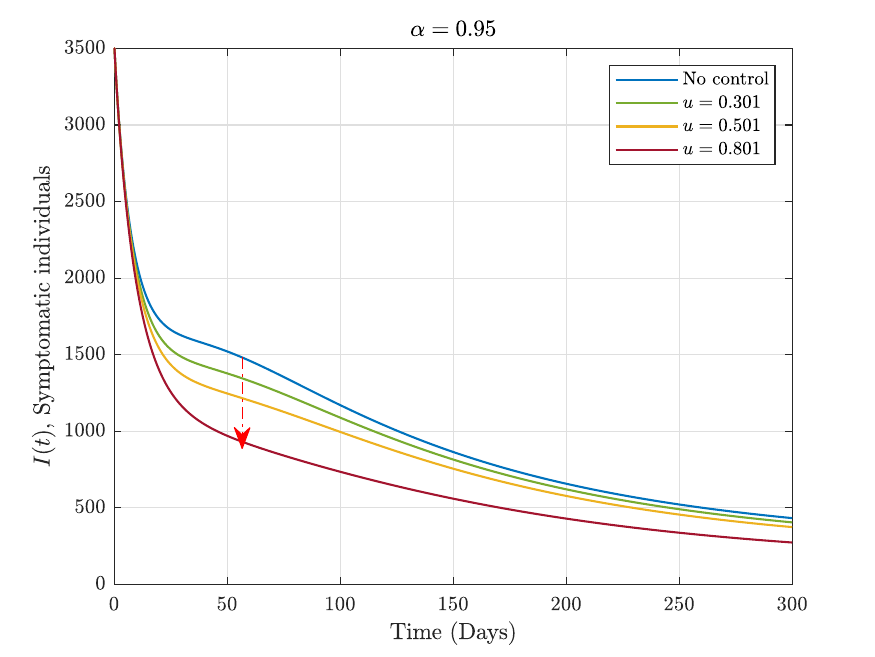}
\caption{}\label{F10}
\end{subfigure}
\begin{subfigure}{.4\textwidth}
\centering
\includegraphics[scale=0.5]{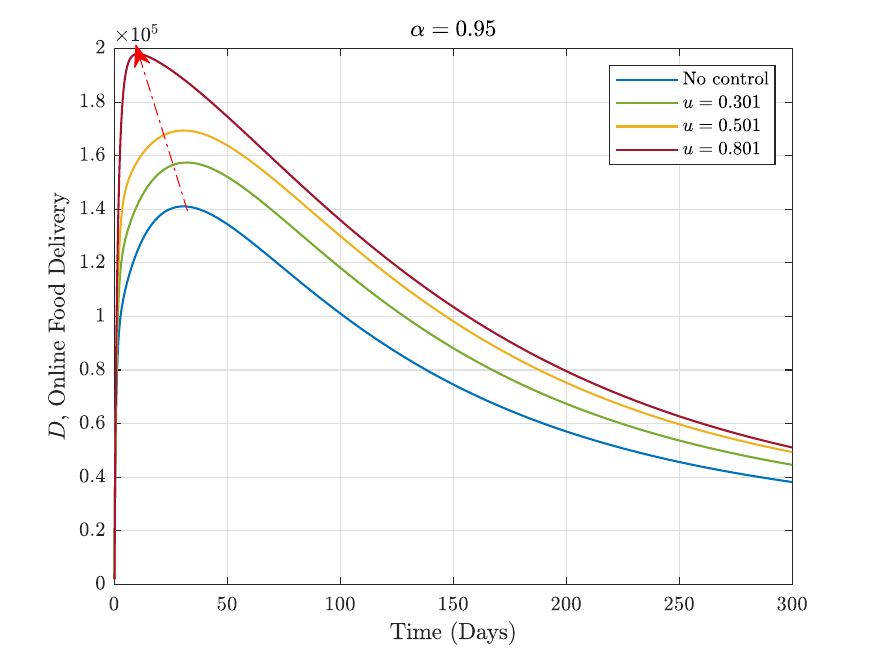}
\caption{}\label{F11}
\end{subfigure}
\caption{Numerical comparison of government control when $\alpha=0.95$.}\label{C}
\end{figure}
% ---------------------------------
\begin{figure}[!ht]
\centering
\includegraphics[scale=0.5]{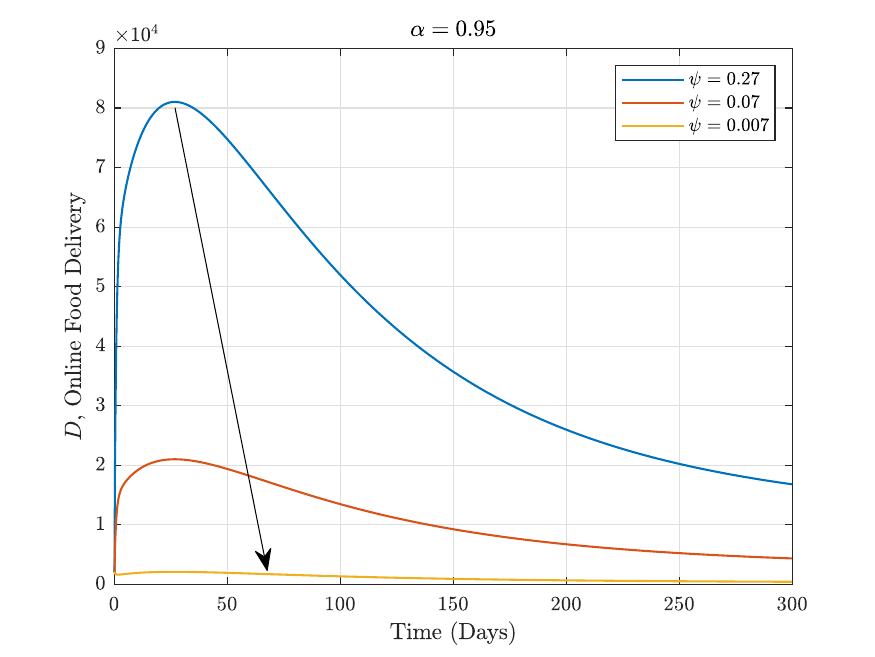}
\includegraphics[scale=0.5]{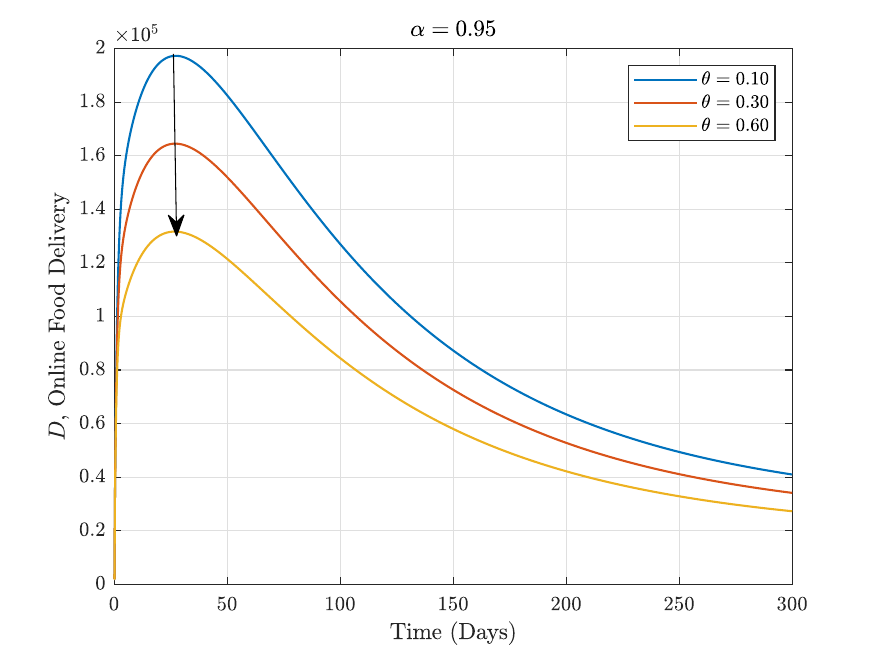}
\caption{Numerical comparison of $\psi$ and $\theta$ when $\alpha=0.95$.}
\label{D}
\end{figure}

% ---------------------------------

\section{Discussion}
\label{sec:9} 

The fractional differential model provides
the best fits to study such epidemic scenarios because they give us more options
to replicate the structure, both theoretically and practically. Fig.~\ref{AA} and
Fig.~\ref{B} represent the numerical simulation results for the human and flies populations.
It is clear that the outcomes and the changes of the fractional-order $\alpha$
fit well, which indicate that the method is effective, thus when the operator
varies by changing $\alpha$, the dynamism of each state variable have the same trend.
From Fig.~\ref{F1}--\ref{F4}, we see that when the fractional order $\alpha$ is reduced
to 1, all the human compartments increase significantly. The biological or
real-life meaning of this dynamics is that urbanization,
busy schedules, and a lack of time to prepare meals, as well as the convenience
of the process, time-saving benefits, and accessibility to information on all
food types, menus, and prices, all contribute to the rising popularity of
online food ordering and delivery in numerous nations. The profit or benefit of
food delivery comes however with a high risk of food safety challenge to the already
existing food safety challenges of the ready-to-eat food, which leads to high
susceptibility and infectivity of food-borne diseases. In Fig.~\ref{F5}--\ref{F7},
we see that the dynamics of flies development goes through
four stages in the aquatic environment, including egg, larvae, pupae, and adult
fly population. Any control intervention, at any stage, results in the management
of adult flies in the area under consideration. The flies mature in this filthy
habitat and become free to wander around the neighbouring housing communities,
where they spread the bacteria that cause various food-borne diseases.
In Fig.~\ref{F8}--\ref{F11}, we show the numerical sensitivity of food-borne diseases
transmission for humans under the Caputo fractional operator with order
$\alpha=0.95$, when the rate of government intervention coverage is varied. We
observe that as the rate of government intervention coverage increases, there
is a decline in the number of asymptomatic and symptomatic infected humans, and
a little increase in the number of susceptible humans and online food delivery.
This result gives chance to the government to implement any food safety control
intervention to the maximum since it will not affect the order-delivery food
but reduce infection. In Fig.~\ref{D}, we show the numerical sensitivity of
$\psi$ and $\theta$ on online food delivery under the Caputo fractional
operator with order $\alpha=0.95.$ We note that as the rate of $\psi$ and
$\theta$  increase, there is a high decrease in the online food delivery
compartment of the model. From a biological point of view, these observations
suggest that the memories of environmental hygiene and disease inflow may
reduce food-borne diseases transmission.

% -----------------------------------------------

\section{Conclusion}
\label{sec:10}

We have proposed and comprehensively analyzed a novel deterministic mathematical 
model for the transmission of food-borne diseases in a population 
consisting of humans and flies, considering the Caputo fractional
order derivative and employing a Predictor-Corrector scheme. The qualitative
aspects of the model, including positivity, boundedness, equilibrium points,
and the basic reproductive number, have been thoroughly investigated.
The analysis of system existence and uniqueness has been carried out using
the Banach and Schauder fixed point theorems. Additionally, the stability
of the fractional controlling system of equations has been examined using
Hyers-Ulam-type stability criteria. To assess the efficacy of the proposed
fractional-order model, numerical trajectories have been generated.
We have also examined the impact of critical parameters. Based on these trajectories,
we have hypothesized that the memory index or fractional order could be utilized
by public health policymakers to comprehend and predict the dynamics of food-borne
disease transmission. Furthermore, we have observed that maximum implementation
of food safety control interventions by the government may not have an effect
on food delivery services. Additionally, we have found that improvements in environmental
hygiene and the reduction of disease inflow can potentially decrease
the incidence of food-borne diseases.

% -----------------------------------------------

\section*{Acknowledgements}

This work was partially supported by the Funda\c{c}\~{a}o
para a Ci\^{e}ncia e a Tecnologia, I.P. (FCT, Funder ID = 50110000187)
under grants UIDB/04106/2020 and UIDP/04106/2020 (CIDMA);
and project 2022.03091.PTDC (CoSysM3).

\section*{Author contributions}

The authors confirm sole responsibility for study conceptualization, 
methodology, formal analysis, codes writing, manuscript writing 
and editing, and approval of the final manuscript.

\section*{Declarations}

\subsection*{Funding}

This work was partially supported by the Funda\c{c}\~{a}o
para a Ci\^{e}ncia e a Tecnologia, I.P. (FCT, Funder ID = 50110000187)
under grants UIDB/04106/2020 and UIDP/04106/2020 (CIDMA);
and project 2022.03091.PTDC (CoSysM3).

\subsection*{Competing interests}

The authors have no competing interests to declare 
that are relevant to the content of this article.

\subsection*{Availability of data and materials}

The authors confirm that the data supporting the findings 
of this study are available within the article.

\subsection*{Code availability}

Matlab code is available from the corresponding 
author on reasonable request.

% -----------------------------------------------

% -----------------------------------------------

\end{document}